\documentclass[12pt]{amsart}
\usepackage{amsfonts, amssymb, latexsym, tikz}

\newtheorem{thm}{Theorem}[section]
\newtheorem{cor}[thm]{Corollary}

\newtheorem{example}[thm]{Example}
\newtheorem{proposition}[thm]{Proposition}

\theoremstyle{remark}

\theoremstyle{definition}

\usepackage{tikz}
\usetikzlibrary{arrows.meta}
\usetikzlibrary{shapes.geometric}

\newcounter{indice}

\newcommand{\cyclefig}[1]{\begin{tikzpicture}[scale= 0.35]
        \cycle{#1}
      \end{tikzpicture} \hspace{2.7mm}\vspace{0.9mm}}

\newcommand{\cycle}[1]{ 
  \setcounter{indice}{0};
  \foreach \i in {#1}
  \addtocounter{indice}{1};
  \addtocounter{indice}{1}
  \draw [dashed, very thin] (0,0) grid (\theindice-1,\theindice-1);
  \setcounter{indice}{1};
  \foreach \i in { #1 } { 
  \draw[draw=white, very thick, double=black] (\theindice-.5, \theindice-.5)--(\theindice-.5, \i-.5);
  \draw (\i-.5, \i-.5)--(\theindice-.5, \i-.5);
  \draw (\theindice-.5,\i-.5) [fill] circle (.18);
  \addtocounter{indice}{1};
  }
  \addtocounter{indice}{-1};
}

\DeclareMathOperator{\cyc}{cyc}
\DeclareMathOperator{\fl}{fl}
\DeclareMathOperator{\td}{td}

\begin{document}

\title{The shallow permutations are the unlinked permutations}

\author{Alexander Woo}
\address{Department of Mathematics, University of Idaho, P.O. Box 441103,
Moscow, ID 83844-1103}
\email{awoo@uidaho.edu}
\thanks{AW was partially supported by Simons Collaboration Grant 359792.}

\subjclass{05A05;57K10}

\keywords{}

\date{}

\begin{abstract}
Diaconis and Graham studied a measure of distance from the identity in the symmetric group called total displacement and showed that it is bounded below by the sum of length and reflection length.  They asked for a characterization of the permutations where this bound is an equality; we call these the shallow permutations.  Cornwell and McNew recently interpreted the cycle diagram of a permutation as a knot diagram and studied the set of permutations for which the corresponding link is an unlink.  We show the shallow permutations are precisely the unlinked permutations.  As Cornwell and McNew give a generating function counting unlinked permutations, this gives a generating function counting shallow permutations.
\end{abstract}

\maketitle

\section{Introduction}

There are many measures for how far a given permutation $w\in S_n$ is from being the identity.  The most classical are length and reflection length, which are defined as follows.  Let $s_i$ denote the adjacent transposition $s_i=(i\,\,i+1)$ and $t_{ij}$ the transposition $t_{ij}=(i\,\,j)$.  The {\bf length} of $w$, denoted $\ell(w)$, is the smallest integer $\ell$ such that there exist indices $i_1,\ldots,i_\ell$ with $w=s_{i_1}\cdots s_{i_\ell}$.  It is classically known that the length of $w$ is equal to the number of inversions of $w$; an {\bf inversion} is a pair $(a,b)$ such that $a<b$ but $w(a)>w(b)$.  The {\bf reflection length} of $w$, which we will denote $\ell_T(w)$, is the smallest integer $r$ such that there exist indices $i_1,\ldots,i_r$ and $j_1,\ldots,j_r$ with $w=t_{i_1j_1}\cdots t_{i_rj_r}$.  It is classically known that $\ell_T(w)$ is equal to $n-\cyc(w)$, where $\cyc(w)$ denotes the number of cycles in the cycle decomposition of $w$.

Another such measure is {\bf total displacement}, defined by Knuth~\cite{KnuAOCP} as $\td(w)=\sum_{i=1}^n |w(i)-i|$ and first studied by Diaconis and Graham~\cite{DG} under the name Spearman's disarray.  Diaconis and Graham showed that $\ell(w)+\ell_T(w)\leq\td(w)$ for all permutations $w$ and asked for a characterization of those permutations for which equality holds.  More recently, Petersen and Tenner~\cite{PT} defined a statistic they call {\bf depth} on arbitrary Coxeter groups and showed that, for any permutation, its total displacement is always twice its depth.  Following their terminology, we call the permutations for which the Diaconis--Graham bound is an equality the {\bf shallow} permutations.

\begin{figure}[hbtp]
\cyclefig{7,5,6,3,4,2,1}
\caption{\label{diag:7563421}Knot diagram for $w=7563421$}
\end{figure}

In a recent paper, Cornwell and McNew~\cite{CM} interpreted the cycle diagram of a permutation as a knot diagram and studied the permutations whose corresponding knots are the trivial knot or the trivial link.  Given a permutation $w$, to obtain the {\bf cycle diagram}, draw a horizontal line between the points $(i,i)$ and $(w^{-1}(i),i)$ for each $i$ and a vertical line between $(j,j)$ and $(j,w(j))$ for each $j$.  Turn the cycle diagram into a {\bf knot diagram} by designating every vertical line to cross over any horizontal line it meets.  For example, Figure~\ref{diag:7563421} shows the knot diagram for $w=7563421$.  They say that a permutation is {\bf unlinked} if the knot diagram of the permutation is a diagram for the unlink, a collection of circles embedded trivially in $\mathbb{R}^3$.  In their paper, they mainly consider derangements, but it is easy to modify their definitions to consider all permutations by treating each fixed point as a tiny unknotted loop.

Our main result is the following:

\begin{thm}
A permutation is shallow if and only if it is unlinked.
\end{thm}

Readers can check that Figure~\ref{diag:7563421} shows that the diagram of $w=7563421$ is a diagram of the unlink with 2 components, and $\ell(w)=19$, $\ell_T(w)=5$, and $\td(w)=24$, so $\ell(w)+\ell_T(w)=\td(w)$.

Using this theorem and further results of Cornwell and McNew~\cite[Theorem 6.5]{CM}, we obtain a generating function counting shallow permutations.  Let $P$ be the set of shallow permutations, and let
$$G(x)=\sum_{n=0}^\infty \sum_{P\cap S_n} x^n.$$  Then $G$ satisfies the following recurrence.

\begin{cor}
The generating function $G$ satisfies the following recurrence:
$$x^2G^3 + (x^2 - 3x + 1)G^2 + (3x-2)G + 1 =0.$$
\end{cor}

This is sequence A301897 (defined as the number of shallow permutations) in the OEIS~\cite{OEIS}.

While this paper was being prepared, Berman and Tenner~\cite{BM21} gave another characterization of shallow cycles that could also be compared with the work of Cornwell and McNew to give our results.

Our proof relies on a recursive description of the set of unlinked permutations due to Cornwell and McNew and a different recursive description of the set of shallow permutations due to Hadjicostas and Monico~\cite{HM}.  We show by induction that all permutations satisfying the description of Cornwell and McNew are shallow and separately that all permutations satisfying the description of Hadjicostas and Monico are unlinked.

The shallow permutations have another surprising connection not previously noted in the literature.  Given a permutation $w$, Bagno, Biagioli, Novick, and the last author~\cite{BBNW} defined the {\bf reduced reflection length} $\ell_R(w)$ as the smallest integer $q$ such that there exist $i_1,\ldots,i_q$ and $j_1,\ldots,j_q$ such that $w=t_{i_1j_1}\cdots t_{i_qj_q}$ and $\ell(w)=\sum_{k=1}^q \ell(t_{i_kj_k})$ and show that the shallow permutations are equivalently the permutations for which $\ell_T(w)=\ell_R(w)$.  Bennett and Blok~\cite{BB} show, using somewhat different language, that reduced reflection length is the rank function on the universal Grassman order introduced by Bergeron and Sottile~\cite{BS} to study questions in Schubert calculus.

Section 2 describes the recursive characterizations of Cornwell and McNew and of Hadjicostas and Monico, while the proof of our main theorem is given in Section 3.

I originally conjectured Theorem 1.1 out of work on a related conjecture in an undergraduate directed research seminar in Spring 2019.  I thank the students in the seminar, specifically Jacob Alderink, Noah Jones, Sam Johnson, and Matthew Mills, for ideas that helped spark this work.  I also thank Nathan McNew for the Tikz code to draw the figures.  Finally, I learned about the work of Cornwell and McNew at Permutation Patterns 2018 and thank the organizers of that conference.

\section{Characterizations of shallow and unlinked permutations}

We now describe the recursive characterizations of unlinked and shallow permutations.

Let $w\in S_n$ be a permutation.  Denote by $\fl_i(w)$ the {\bf $i$-th flattening} of $w$, which is defined by removing the $i$-th entry of $w$ (in one-line notation) and then renumbering down by 1 every entry greater than $w(i)$.
Formally, $$\fl_i(w)(k)=\begin{cases}
w(k) &\mbox{if } k<i \mbox{and } w(k)<w(i) \\
w(k)-1 &\mbox{if } k<i \mbox{and } w(k)>w(i) \\
w(k+1) &\mbox{if } k>i \mbox{and } w(k)<w(i) \\
w(k+1)-1 &\mbox{if } k>i \mbox{and } w(k)>w(i) \\
\end{cases}$$

Cornwell and McNew~\cite{CM} give the following recursive characterization of permutations with unlinked cycle diagrams.

\begin{thm}
Suppose $w$ is unlinked.  Then either
\begin{itemize}
\item $w\in S_1$ (so $w=1$ in one-line notation), OR
\item There exists $i$ with $|w(i)-i|\leq 1$, and $\fl_i(w)$ is unlinked.
\end{itemize}
\end{thm}

This characterization is assembled from several statements in their paper, and we consider all permutations instead of only derangements, so we explain how to obtain this statement from their work.  References to specific statements are by the numbering in~\cite{CM}

\begin{proof}
Suppose $w\in S_n$ is unlinked.  If $w(i)=i$ for some $i$, then $|w(i)-i|=0$ and $\fl_i(w)$ is unlinked.  This handles the case where $w$ has a fixed point.

Applying Lemma 6.3 repeatedly until some $\tau_i$ is a single cycle, we see that $w$ has some cycle involving the consecutive entries $j, j+1, \ldots, k$.  Now Proposition 5.10 applied to this cycle shows that there is some index $i$ with $j\leq i\leq k$ such that $|w(i)-i|=1$.  The process of going from the diagram $D$ to the diagram $D_0$ described in the second paragraph of the proof of Proposition 5.11 is precisely $\fl_i$.
\end{proof}

\begin{example}
Let $w=7563421$.  Then $w(4)=3$, so $|w(4)-4|=1$.  Furthermore, $\fl_4(w)=645321$, which is also unlinked.
\end{example}

Given a permutation $w\in S_n$, an index $j$ is a {\bf left-to-right maximum} if $w(j)>w(i)$ for all $i<j$.  An index $j$ is a {\bf right-to-left minimum} if $w(j)<w(i)$ for all $i>j$.

Hadjicostas and Monico~\cite[Theorem 4.1]{HM} give the following recursive characterization of shallow permutations.

\begin{thm}
\label{thm:shallow}
Suppose $w\in S_n$ is shallow.  Then either
\begin{itemize}
\item $w\in S_1$ (so $w=1$ in one line notation), OR
\item $w(n)=n$, and the permutation $w'\in S_{n-1}$ with $w'(i)=w(i)$ for all $i$ is shallow, OR
\item $w(n)=k$, $w^{-1}(n)=j$, and the permutation $w'\in S_{n-1}$ defined by setting $w'(i)=w(i)$ for $i\neq j$ and $w'(j)=k$ is shallow with either a left-to-right maximum or right-to-left minimum at $j$.
\end{itemize}
\end{thm}

\begin{example}
If $w=7563421$, then $w'=156342$ is shallow with both a left-to-right maximum and a right-to-left minimum at position $1$.  If $w=45231$, then $w'=4123$ is shallow with a right-to-left minimum at position $2$.
\end{example}

\section{Proof of Main Theorem}

To prove our main theorem, we use the two recursive characterizations.  We split the proof into two parts, first using the characterization of Cornwell and McNew to prove the fllowing.

\begin{proposition}
Every unlinked permutation is shallow.
\end{proposition}

\begin{proof}
We prove this proposition by induction on $n$.  Let $w$ be an unlinked permutation.

For the base case, clearly $\ell(w)+\ell_T(w)=\td(w)$ for the permutation $w=1$.  (Both sides are 0.)

For the inductive case, suppose there exists $i$ with $|w(i)-i|\leq 1$ and $\fl_i(w)$ unlinked.  Given integers $a$ and $b$, let $a'=a$ if $a<i$ and $a'=a-1$ if $a>i$,
and similarly $b'=b$ if $b<i$ and $b'=b-1$ if $b>i$.  Then note that for $a,b\neq i$, $(a,b)$ is an inversion of $w$ if and only if
$(a',b')$ is an inversion of $\fl_i(w)$.  Hence $\ell(w)-\ell(\fl_i(w))$ is equal to the number of inversions involving $i$, or, in notation, the number of pairs $(a,i)$ with $a<i$ and $w(a)>w(i)$ and pairs $(i,b)$ with $i<b$ and $w(i)>w(b)$.

We now split into three cases depending on whether $w(i)-i$ is $0$, $1$, or $-1$.

If $w(i)-i=0$, then $\ell_T(\fl_i(w))=\ell_T(w)$, as $\fl_i(w)$ has one fewer cycle, namely the fixed point $i$ that was removed, and $\fl_i(w)$ is a permutation of one fewer element.
Furthermore, since $w(i)=i$, $|\fl_i(w)(a')-a'|=|w(a)-a|$ if and only if $(a,i)$ or $(i,a)$ is not an inversion of $w$, and $|\fl_i(w)(a')-a'|=|w(a)-a|-1$ if it is an inversion.  (Note that this is
so simple because the sign of $w(a)-a$ is determined by whether $(a,i)$ is an inversion or $(i,a)$ is an inversion.)  Also $w(i)-i=0$.  Hence $\ell(w)-\ell(\fl_i(w))=\td(w)-\td(\fl_i(w))$.

By the inductive hypothesis we can assume $\ell(\fl_i(w))+\ell_T(\fl_i(w))=\td(\fl_i(w))$, so $\ell(w)+\ell_T(w)=\td(w)$.

If $w(i)-i=-1$, then the cycle decomposition of $\fl_i(w)$ is the same as that of $w$ except that $i$ is removed and every $b>i$ is replaced by $b-1$.  (In particular, $\fl_i(w)(w^{-1}(i))=w(i)=i-1$.)  Hence $\ell_T(\fl_i(w))=\ell_T(w)-1$.  Furthermore, also in this case, $|\fl_i(w)(a')-a'|=|w(a)-a|$ if and only if $(a,i)$ or $(i,a)$ is not an inversion of $w$, and $|\fl_i(w)(a')-a'|=|w(a)-a|-1$ if it is an inversion.  However, $|w(i)-i|=1$, so $\td(w)-\td(\fl_i(w))=\ell(w)-\ell(\fl_i(w))+1$.

Again by the inductive hypothesis we can assume $\ell(\fl_i(w))+\ell_T(\fl_i(w))=\td(\fl_i(w))$, so $\ell(w)+\ell_T(w)=\td(w)$.

The proof where $w(i)-i=1$ is similar to the previous case, as again we have $\ell_T(\fl_i(w))=\ell_T(w)-1$ and $\td(w)-\td(\fl_i(w))=\ell(w)-\ell(\fl_i(w))+1$.
\end{proof}

\begin{example}
Let $w=7563421$, and let $i=4$, so $w(i)-i=-1$.  Then$\fl_4(w)=645321$, with $\ell_T(\fl_4(w))=4$.  Furthermore, $\td(w)-\td(\fl_4(w))=6$, and $\ell(w)-\ell(\fl_4(w))=5$.
\end{example}

We now follow the recursive characterization of Hadjicostas and Monico to prove the following:

\begin{proposition}
Every shallow permutation is unlinked.
\end{proposition}

\begin{proof}
We prove this by induction on $n$.

If $w\in S_1$, then the associated link is a single small unknotted and unlinked loop.

If $w(n)=n$, $w'\in S_{n-1}$ is defined by $w'(i)=w(i)$ for all $i$ with $1\leq i\leq n-1$, and $w'$ is unlinked, then the cycle diagram of $w$ is obtained from that of $w'$ by adding a small unknotted and unlinked loop at the top right, so it is also unlinked.

Now suppose $w(n)=k$, $w^{-1}(n)=j$, $w'$ as defined in Theorem~\ref{thm:shallow} is shallow, and $w'(j)=k$ is a right-to-left minimum.  The cycle diagram of $w$ can be obtained from the cycle diagram of $w'$ by deleting the vertical segment
from $(j,j)$ to $(j,k)$ and replacing it with segments from $(j,j)$ to $(j,n)$ to $(n,n)$ to $(n,k)$ to $(j,k)$.  Since $(j,k)$ is a right-to-left minimum in $w'$, the only crossings made by the new
segments are on the vertical segment from $(j,j)$ to $(j,n)$.  Since they are on a vertical segment, these are all overcrossings.  Hence this long loop in the link associated to $w$ can be
slid around over the top of the knot and shrunk to the vertical segment from $(j,j)$ to $(j,k)$, which also only has overcrossings.  Therefore, the link types of $w$ and $w'$ are the same.
By induction, $w'$ is unlinked, so $w$ is also unlinked.

One has a similar argument if $w'(j)=k$ is a left-to-right maximum, except that the crossings are undercrossings associated to horizontal segments and hence the isotopy takes place under the rest of the link.  If $w'(j)=k$ is both a left-to-right maximum and a right-to-left minimum, then $j=k$ and the new segments make no crossings at all, forming a free unknotted link component.
\end{proof}

\begin{example}
Let $w=7563421$.  Here we have $k=1$, $j=1$, and $j=k$ is both a left-to-right maximum and a right-to-left minimum.  One can see that the cycle $(17)$ produces a free unknotted link component that can be shrunk to a little loop at $1$.

\begin{figure}[htbp]
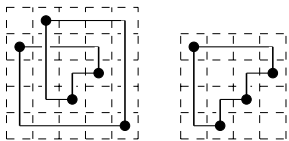

\cyclefig{4,5,2,3,1}\,\,\cyclefig{4,1,2,3}
\caption{\label{diag:45231isotopy}Knot diagrams for $w=45231$ and $w'=4123$}
\end{figure}

Now let $w=45231$.  Then $k=1$, $j=2$, and $w'=4123$.  One can see from Figure~\ref{diag:45231isotopy} that the knot diagrams for $w=45231$ and $w'=4123$ are isotopic as described above.
\end{example}


\begin{thebibliography}{XXXX}
\bibitem{BB} Curtis D. Bennett and Rieuwert J. Blok, {\it Partial orders generalizing the weak order on Coxeter groups}. J. Combin. Theory Ser. A {\bf 102} (2003), 331--346.
\bibitem{BS} Nantel Bergeron and Frank Sottile, {\it Schubert polynomials, the Bruhat order, and the geometry of flag manifolds}.  Duke Math J. {\bf 95} (1998), 373--423.
\bibitem{BM21} Yosef Berman and Bridget Tenner, {\it Permutation-functions, statistics, and shallow permutations}.  Preprint available at arXiv:2110.11146.
\bibitem{BBNW} Eli Bagno, Riccardo Biagioli, Mordechai Novick, and Alexander Woo, {\it Depth in classical Coxeter groups}.  J. Algebraic Combin. {\bf 44} (2016), 645--676.
\bibitem{CM} Christopher Cornwell and Nathan McNew, {\it Unknotted cycles}.  Preprint available at arXiv:2007.04917.
\bibitem{DG} Persi Diaconis and Ronald L. Graham, {\it Spearman's Footrule as a Measure of Disarray}. J. R. Stat. Soc. Ser. B. Stat. Methodol. {\bf 39} (1977), 262--268.
\bibitem{HM} Petros Hadjicostas and Chris Monico, {\it A re-examination of the Diaconis-Graham inequality}. J. Combin. Math. Combin. Comput. {\bf 87} (2013), 275-295.
\bibitem{KnuAOCP} Donald E. Knuth, {\it The art of computer programming, Vol. 3: Sorting and Searching, 2nd edition}.  Addison-Wesley, Reading, MA, 1998. 
\bibitem{OEIS} OEIS Foundation Inc. (2020).  The On-Line Encyclopedia of Integer Sequences, http://oeis.org/A301897
\bibitem{PT} T. Kyle Petersen and Bridget E. Tenner, {\it The depth of a permutation}.  J. Comb. {\bf 6} (2015) 145--178.
\end{thebibliography}
\end{document}